\documentclass{amsart}
\usepackage{amssymb}

\newtheorem{theorem}{Theorem}
\newtheorem{lemma}{Lemma}

\newtheorem{corollary}{Corollary}

\newtheorem{remark}{Remark}

\newcommand\innerprod[2]{\langle #1,#2 \rangle}
\newcommand{\norma}[1]{|\!| #1 |\!|}
\newcommand{\eps}{\varepsilon}

\newcommand{\Rn}[1]{\mathbb{R}^{#1}}
\begin{document}

\title{Shadows of cube vertices}
\markright{shadows of cube vertices}
\author{Yossi Lonke}

\begin{abstract}
Suppose that a finite-dimensional cube is orthogonally projected onto a central section of itself by a subspace of one dimension less. Up to dimension $9$, at least one vertex is projected onto the section,
but for dimension $10$ or larger, there are orthogonal projections for which all the vertices are projected outside the section. In fact, this is the case for ``most" orthogonal projections, as the dimension tends to infinity.
\end{abstract}

\maketitle

If the vertices of a square or a three-dimensional cube are projected onto a line (respectively, plane) through its center then at least one pair of opposite vertices
is projected onto the intersection with the line (respectively, plane). \emph{What is the situation for higher dimensions?} The purpose of this note is to answer this question.

For a unit vector $u\in \Rn{n}$, let $H_u$ denote the $(n-1)$-dimensional subspace of all vectors in $\Rn{n}$ orthogonal to $u$, and $\pi_u$ the
orthogonal projection onto $H_u$. Thus for each $x\in \Rn{n}$, 

\begin{equation}\label{projection}
\pi_u(x)=x-\innerprod{x}{u}u,
\end{equation}
where $\innerprod{\cdot}{\cdot}$ is the standard inner product.  The three standard norms on $\Rn{n}$ are denoted by
$$\norma{u}_1 = \sum_{k=1}^n|u_k|,\enspace \norma{u}_2=\left(\sum_{k=1}^nu_k^2\right)^{1/2},\enspace\norma{u}_{\infty}=\max_{1\leq k\leq n}|u_k|,$$
 and the cube $[-1,1]^n$ is denoted by $C_n$.

\begin{theorem} 
Let $u\in\Rn{n}$ be a unit vector not orthogonal to any of the vertices of the cube $C_n$, and $\pi_u$ the orthogonal projection  onto $H_u$.
A necessary and sufficient condition that there exists a vertex $\eps\in C_n$ such that $\pi_u(\eps)\in H_u\cap C_n$ is  ${\norma{u}_1\norma{u}_{\infty}\leq 2}$. 
 \end{theorem}
 It turns out (see Lemma 1 below) that the maximal value of $\norma{u}_1\norma{u}_{\infty}$ on the unit sphere $\mathbb{S}^{n-1}$ is $\frac{\sqrt{n}+1}{2}$. Thus the
 condition of the theorem is fulfilled \emph{for all} $u\in\mathbb{S}^{n-1}$ if and only if $n\leq 9$.
 \begin{corollary} Let $n\geq 1$ be an integer. The following are equivalent:
 \begin{enumerate}
 \item For every orthogonal projection $\pi_u$ onto the $(n-1)$-dimensional subspace $H_u$ of $\Rn{n}$, there exists  a vertex $\eps$ of the cube $C_n$ such that $\pi_u(\eps)\in H_u\cap C_n$.
 \item $n\leq 9$.
 \end{enumerate}
 \end{corollary}
 \begin{proof}

By symmetry, it can and will be assumed that the nonzero coordinates of $u$ are positive.
\emph{The condition is necessary}: \enspace choose a vertex $\eps=(\eps_1,\dots,\eps_n)$ for which ${\norma{\pi_u(\eps)}_{\infty}\leq 1}$, and assume $\innerprod{\eps}{u}>0$ (otherwise, replace $\eps$ by $-\eps$). Let $\{e_k\}_{k=1}^n$ denote the standard unit vectors in $\Rn{n}$. Consider the $k$th coordinate of the projection,
\begin{equation}
\innerprod{\pi_u(\eps) }{e_k}=\eps_k-\innerprod{\eps}{u}u_k.
\end{equation}
Pick an index $k$ such that $u_k>0$. If $\eps_k=-1$, then 
\begin{equation}\label{positiveSameSign}
\innerprod{\pi_u(\eps) }{e_k}=\eps_k-\innerprod{\eps}{u}u_k<-1,
\end{equation}
but $\norma{\pi_u(\eps)}_{\infty}\leq 1$, so  $u_k>0$ implies $\eps_k=1$. Hence,
 \begin{equation}\label{theCoordinates}
\innerprod{\pi_u(\eps)}{e_k}=\left\{
\begin{array}{lr}
\pm 1 & \text{if } u_k=0,\\
1-u_k\sum_{i:u_i > 0}u_i & \text{if } u_k>0.
\end{array}
\right.
\end{equation}
As $u_k\geq 0$ for all $k$, the sum of the nonzero coordinates of $u$ is just  $\norma{u}_1$,  and $\norma{u}_{\infty}=\max\{u_k: u_k>0\}$.  Thus, (\ref{theCoordinates}) implies
\begin{equation}\label{normOfProjection}
\norma{\pi_u(\eps)}_{\infty}=\max_k|\innerprod{\pi_u(\eps)}{e_k}|=\max\{1,|1-\norma{u}_{\infty}\norma{u}_1|\},
\end{equation}
and since $\norma{\pi_u(\eps)}_{\infty}\leq 1$, this implies   $\norma{u}_1\norma{u}_{\infty}\leq 2$.

\emph{The condition is also sufficient}, because given a unit vector
we can choose a vertex~$\eps$ for which $\eps_k=1$ for all $k$ such that $u_k>0$. If in addition we assume that $u$ is not orthogonal to any vertex, then (\ref{theCoordinates}), and consequently (\ref{normOfProjection}), are valid. The latter, combined with the condition $\norma{u}_1\norma{u}_{\infty}\leq 2$,  implies  that $\norma{\pi_u(\eps)}_{\infty}\leq 1$.
The proof of the theorem is complete. 
\end{proof}
The following lemma was mentioned above. Curiously, the bound of $2$ appearing in the condition of the theorem is precisely the maximum of the corresponding function for $n=9$.
\begin{lemma}
For every $n\geq 1$, the maximum of the function $f(u)= \norma{u}_1\norma{u}_{\infty}$ on the unit sphere $\mathbb{S}^{n-1}$ is $\frac{\sqrt{n}+1}{2}$. 
\end{lemma}
\begin{proof}
For reasons of symmetry the desired maximum coincides with the maximum of the  function
$$g(u)=u_1(u_1+u_2+\cdots +u_n)$$
subject to the constraint $u\in\mathbb{S}^{n-1}$. The rest of the proof is a standard exercise in multivariable caclulus, and is therefore omitted.
\end{proof}

\begin{remark}

For sufficiently large $n$, most points $u\in \mathbb{S}^{n-1}$, in the sense of  spherical measure, do not satisfy $\norma{u}_1\norma{u}_{\infty}\leq 2$. In fact, in a sense that can be made precise using classical concentration of measure techniques, (see \cite[(2.6)]{FLM})  the ``typical" order of magnitude of  the product $\norma{u}_1\norma{u}_{\infty}$ in $\mathbb{S}^{n-1}$is $\sqrt{\log n}$. Moreover, those vectors that are orthogonal to some vertex of the cube consist of a union of $2^{n-1}$ $(n-2)$-dimensional spheres, which do not occupy any positive measure, so for sufficiently large dimensions, ``most" one-codimensional orthogonal projections carry all the cube's vertices outside the section of the cube by the subspace being projected upon. \end{remark}

yossilonke@me.com

\vfill\eject

\end{document}